\documentclass[11 pt]{amsart}

\usepackage{lineno,hyperref}
\usepackage{amsmath,amssymb}
\usepackage{lineno}
\usepackage{centernot}
\theoremstyle{plain}
\newtheorem{thm}{Theorem}[section]
\newtheorem{lemma}[thm]{Lemma}
\newtheorem{cor}[thm]{Corollary}
\newtheorem{prop}[thm]{Proposition}

\theoremstyle{definition}
\newtheorem{defn}[thm]{Definition}
\newtheorem{example}[thm]{Example}

\numberwithin{equation}{section}

\newcommand{\interior}[1]{{\kern0pt#1}^{\mathrm{o}}}

\begin{document}


\title{Pointwise semicommutative rings}

\author[Sanjiv Subba]{Sanjiv Subba $^\dagger$}

\address{$^\dagger$Department of Mathematics\\ National Institute Of Technology Meghalaya\\ Shillong 793003\\ India}
\email{sanjivsubba59@gmail.com}
\author[Tikaram Subedi]{Tikaram Subedi  {$^{\dagger *}$}}
\address{$^{\dagger * }$Department of Mathematics\\  National Institute Of Technology  Meghalaya\\ Shillong 793003\\ India}
\email{tikaram.subedi@nitm.ac.in}

\author[A. M. Buhphang]{A. M. Buhphang  {$^{\dagger \dagger }$}}
\address{$^{\dagger\dagger  }$Department of Mathematics\\ North-Eastern Hill University\\ Shillong 793022\\ India}
\email{ardeline@nehu.ac.in}
.


\begin{abstract}
We call a ring $R$ pointwise semicommutative if for any element $a\in R$ either $l(a)$ or $r(a)$ is an ideal of $R$. A class of pointwise semicommutative rings is a strict generalization of semicommutative rings. Since reduced rings are pointwise semicommutative, this paper studies sufficient conditions for pointwise semicommutative rings to be reduced. For a pointwise semicommutative ring $R$, $R$ is strongly regular if and only if $R$ is left SF; $R$ is exchange if and only if $R$ is clean; if $R$ is semiperiodic then $R/J(R)$ is commutative.

\end{abstract}

\maketitle

\section{Introduction}

Throughout this paper, unless otherwise mentioned, all rings considered are associative with identity, $R$ represents a ring, and all modules are unital. For any $w\in R$, the notations $r(w)~(l(w))$ represents the right (left) annihilator of $w$. We write $C(R),~P(R),~ J(R),~N(R)$, $E(R)$, $Z(_RR)$ and $U(R)$ respectively, for the set of all central elements, the prime radical, the Jacobson radical, the set of all nilpotent elements, the set of all idempotent elements, the left singular ideal of $R$ and the group of units of $R$. Recall that $R$ is said to be:
\begin{enumerate}
	\item \textit{reduced} if $N(R)=0$.
	\item \textit{reversible} (\cite{ext}) if $wh=0$ implies $hw=0$ for any $w,~h\in R$.
	\item  \textit{semicommutative} (\cite{ext}) if for each $w\in R$, $r(w)$ is an ideal of $R$.

	\item  \textit{strongly regular} (\cite{qsr})  if for each $w\in R$, $w\in w^2R$.
	
	\item  \textit{left (right) weakly regular} (\cite{wreg})  if $w\in RwRw~ (w\in wRwR)$ for any $w\in R$.
	\item \textit{left (right) quasi duo} (\cite{qsr}) if every maximal left (right) ideal of $R$ is an ideal of $R$.
	
	\quad Let $ME_l(R)=\{e\in E(R)~|~Re$~ is~ a~minimal~left~ideal~of~$R\}$.  $R$ is called \textit{left min-abel} if for any $e\in ME_l(R)$ $re=ere$ for all $r\in R$.  $R$ is called left $MC2$ if $aRe=0$ implies $eRa=0$ for any $a\in R,~e\in ME_l(R)$. According to \cite{nci}, $R$ is said to be $NCI$ if $N(R)=0$ or $N(R)$ contains a non-zero ideal of $R$. $R$ is said to be $NI$ if $N(R)$ is an ideal of $R$, and $R$ is 2-primal if $N(R)=P(R)$. Obviously, NI rings are NCI; nevertheless, the converse is not true (by \cite[Example 2.5]{nci}). $R$ is \textit{directly finite}   if $wh=1$ implies $hw=1$ for any $h,w\in R$.
	
	\quad 
	Let $\Psi:R\rightarrow R$ be an automorphism of $R$. $R[x;\Psi]$ is the ring of polynomials over $R$ with respect to usual polynomial addition and multiplication which is defined by the rule: $xa=\Psi(a)x$. $R[x;\Psi]$ is called skew polynomial ring of $R$.
	
	\quad   Over the past several years, semicommutative rings and their generalizations have been studied extensively by many authors. In a semicommutative ring $R$, both the left and the right annihilator of every element of $R$ are ideals of $R$. This motivates us to investigate a ring $R$ in which either the left or the right annihilator (not necessarily both) of any element of $R$ is an ideal of $R$. This paper studies such a class of rings.
\end{enumerate} 

\section{Main Results}

\begin{defn}
	We call a ring $R$ \emph{pointwise semicommutative} if for any $w\in R$,  either $l(w)$ or $r(w)$ is an ideal of $R$.
\end{defn}

It is evident that semicommutative rings are pointwise semicommutative. However, the following example shows that the converse need not be true.
\begin{example} \label{exmpl}
	Let $R=\left[\begin{array}{{rr}}
		\mathbb{Z}_2 & \mathbb{Z}_2\\
		0 & \mathbb{Z}_2
	\end{array} \right]$ and
	$a_0=\left[\begin{array}{{rr}}
		0 & 0\\
		0 & 0
	\end{array} \right]$, $a_1=\left[\begin{array}{{rr}}
		1 & 0\\
		0 & 1
	\end{array} \right]$, $a_2=\left[\begin{array}{{rr}}
		1 & 1\\
		0 & 1
	\end{array} \right]$, $a_3=\left[\begin{array}{{rr}}
		1 & 0\\
		0 & 0
	\end{array} \right]$, $a_4=\left[\begin{array}{{rr}}
		0 & 1\\
		0 & 0
	\end{array} \right]$, $a_5=\left[\begin{array}{{rr}}
		0 & 0\\
		0 & 1
	\end{array} \right]$,
	$a_6=\left[\begin{array}{{rr}}
		1 & 1\\
		0 & 0
	\end{array} \right]$, $a_7=\left[\begin{array}{{rr}}
		0 & 1\\
		0 & 1
	\end{array} \right]$. Observe that $l(a_0)=R$, $l(a_1)=l(a_2)=0$,  $l(a_3)=\left[\begin{array}{{rr}}
		0 & \mathbb{Z}_2\\
		0 & \mathbb{Z}_2
	\end{array} \right]$, $l(a_4)=\left[\begin{array}{{rr}}
		0 & \mathbb{Z}_2\\
		0 & \mathbb{Z}_2
	\end{array} \right]$, $r(a_5)=\left[\begin{array}{{rr}}
		\mathbb{Z}_2 & \mathbb{Z}_2\\
		0 & 0
	\end{array} \right]$, $l(a_6)=\left[\begin{array}{{rr}}
		0 & \mathbb{Z}_2\\
		0 & \mathbb{Z}_2
	\end{array} \right]$, $r(a_7)=\left[\begin{array}{{rr}}
		\mathbb{Z}_2 & \mathbb{Z}_2\\
		0 & 0
	\end{array} \right]$ are ideals. Thus, $R$ is pointwise semicommutative. Clearly, $R$ is not semicommutative.
\end{example}	

\begin{prop}\label{dlm}
	Let R be a pointwise semicommutative ring. Then:
	\begin{enumerate}
		\item \label{df} R is directly finite.
		\item \label{lma} R is left min-abel.
		
	\end{enumerate}
\end{prop}
\begin{proof}
	\begin{enumerate}
		\item Suppose $w,h\in R$ with $wh=1$. Take $k=h-h^2w$. Since $R$ is pointwise semicommutative and $k^2=0$, $0=k(w-h^2w)k =(h-h^2w)(w-h^2w)(h-h^2w)$. This implies that $hw=1$.
		
		\item  Let $e\in ME_l(R)$ and $w\in R$. Take $h=we-ewe$. If possible, assume that $h\neq 0$. Clearly, $he=h$ and $h^2=0$.  Observe that $0\neq Rh\subseteq Re$. Since $e\in ME_l(R)$, $Rh=Re$. As $R$ is pointwise semicommutative and $h^2=0$, $hRh=0$. Thus, $0=RhRh=ReRe=Re$, a contradiction. Thus, $h=0$.
		
	\end{enumerate}
\end{proof}

\begin{prop}
	Let $R$ be a ring. Then R is a domain if and only if R is prime and pointwise semicommutative.
\end{prop}
\begin{proof}
	The necessary part is obvious. Conversely, assume that $R$ is a prime and pointwise semicommutative ring and $w,h\in R$ be such that $wh=0$. Since $R$ is pointwise semicommutative and $(hw)^2=0$, $hwRhw=0$. By hypothesis, $hw=0$. So for any $r\in R$, $(wrh)^2=0$ which further implies that $wrhRwrh=0$, that is, $wrh=0$. As $R$ is prime, $w=0$ or $h=0$.
\end{proof}

\quad It is well known that the ring  \\ $R_n=
\left \lbrace
\left[\begin{array}{lccccr}
	a & a_{12} & \dots  & a_{1(n-1)} & a_{1n}\\
	0 & a & \dots & a_{2(n-1)} & a_{2n}\\
	\vdots & \vdots &\ddots & \vdots & \vdots \\
	0 & 0 & \dots & a &a_{(n-1)n}\\
	0 & 0 & \dots & 0 & a \\
\end{array}
\right ]|a, a_{ij}\in R,~ i<j\right \rbrace$
is semicommutative whenever $R$ is reduced and $n=3$. However, $R_n$ is not semicommutative for $n\geq 4$ even if $R$ is reduced (see \cite[Example 1.3]{ext}). So one might suspect whether $R_n$ is pointwise semicommutative for $n\geq 4$ whenever $R$ is reduced. Nevertheless, the following example obliterates the possibility.

\begin{example}\label{R4}
	Let $R=\mathbb{Z}_6$, $R_4=\left\{\left[\begin{array}{rrrr}
		a & b & c & d\\
		0 & a & e & f\\
		0 & 0 & a & g\\
		0 & 0 & 0 & a
	\end{array}\right]|a,b,c,d,e,f,g\in R \right\} $, where $\mathbb{Z}_6$ is the ring of integers modulo 6. Take $A=\left[ \begin{array}{rrrr}
		2 & 1 & 1 & 1\\
		0 & 2 & 2 & 1\\
		0 & 0 & 2 & 1\\
		0 & 0 & 0 & 2
	\end{array}\right]\in R_4$. Note that $B=\left[ \begin{array}{rrrr}
		0 & 3 & 3 & 3\\
		0 & 0 & 0 & 3\\
		0 & 0 & 0 & 3\\
		0 & 0 & 0 & 0
	\end{array}\right]\in r(A)$. Let $C=\left[ \begin{array}{rrrr}
		1 & 5 & 5 & 2\\
		0 & 1 & 3 & 1\\
		0 & 0 & 1 & 5\\
		0 & 0 & 0 & 1
	\end{array}\right] \in R_4$. Then  $ACB=\left[ \begin{array}{rrrr}
		0 & 0 & 0 & 3\\
		0 & 0 & 0 & 0\\
		0 & 0 & 0 & 0\\
		0 & 0 & 0 & 0
	\end{array}\right]\neq 0$. Now, observe that $E=\left[ \begin{array}{rrrr}
		0 & 3 & 3 & 3\\
		0 & 0 & 0 & 3\\
		0 & 0 & 0 & 3\\
		0 & 0 & 0 & 0
	\end{array}\right]\in l(A)$. Take $F=\left[ \begin{array}{rrrr}
		5 & 5 & 5 & 5\\
		0 & 5 & 3 & 5\\
		0 & 0 & 5 & 3\\
		0 & 0 & 0 & 5
	\end{array}\right]\in R_4$.
	Then $EFA=\left[ \begin{array}{rrrr}
		0 & 0 & 0 & 3\\
		0 & 0 & 0 & 0\\
		0 & 0 & 0 & 0\\
		0 & 0 & 0 & 0
	\end{array}\right]\neq 0$. Thus, neither $r(A)$ nor $l(A)$ is an ideal of $R_4$, that is, $R_4$ is not a pointwise semicommutative ring.
\end{example}.

\begin{prop}\label{red}
	Let $R$ be a pointwise semicommutative ring. Then $R$ is NCI.
	
\end{prop}
\begin{proof}
	
	Let $R$ be a pointwise semicommutative ring. Suppose $N(R)\neq 0$. Then there exists $w~(\neq 0)\in N(R)$ such that $w^n=0$ for some integer $n\geq 2$ and $w^{n-1}\neq 0$. Since $R$ is pointwise semicommutative, either $w^{n-1}Rw=0$ or $wRw^{n-1}=0$. So $w^{n-1}Rw^{n-1}=0$. Hence $Rw^{n-1}R$ is a non-zero nilpotent ideal of $R$. Thus, $R$ is $NCI$.

\end{proof}
Observe that $R_4$ (in Example \ref{R4}) is NCI. Hence the converse is not true.

\begin{prop}
	Let $\{R_i \}_{i\in \Delta}$ be a class of rings and $\Delta$ an index set. If $R=\Pi_{i\in \Delta}R_i$ is pointwise semicommutative, then $R_i$ is pointwise semicommutative for each $i\in \Delta$.
\end{prop}
\begin{proof}
	Let $a_j\in R_j,~j\in \Delta$. Suppose $l\left((0,...,0,a_j,0,...)\right)$ is an ideal of $R=\Pi_{i\in \Delta}R_i$ and $x_j\in l(a_j)$. Note that $(0,0,...,0,x_j,0...)\in l\left((0,...,0,a_j,0,...)\right)$. So $(0,0,....,0,x_j,0,..)(r_i)_{i\in \Delta}(0,...,a_j,0,...)=0$ for any $(r_i)_{i\in \Delta} \in \Pi_{i\in \Delta}R_i$. So $x_jr_j\in l(a_j)$ for all $r_j\in R_j$. Thus, $l(a_j)$ is an ideal of $R_j$. Similarly, $r(a_j)$ is an ideal of $R_j$ whenever $r\left((0,...,0,a_j,0,...)\right)$ is an ideal of $R$. Therefore, for each $j\in \Delta$, $R_j$ is pointwise semicommutative.
\end{proof}
However, the converse is not true (see the following example).
\begin{example}
	Let $R_i=\left[\begin{array}{rr}
		\mathbb{Z}_2 & \mathbb{Z}_2 \\
		0 & \mathbb{Z}_2 \\
	\end{array} \right] $, $i\in \{1,2\}$. Then $R_i$ is pointwise semicommutative (see Example \ref{exmpl}). Take $A=\left(\left[\begin{array}{rr}
		0 & 0 \\
		0 & 1 \\
	\end{array} \right], \left[\begin{array}{rr}
		1 & 1 \\
		0 & 0 \\
	\end{array} \right]\right) \in R_1\times R_2$. Then $l(A)=\left[\begin{array}{rr}
		\mathbb{Z}_2 & 0 \\
		0 & 0 \\
	\end{array} \right] \times \left[\begin{array}{rr}
		0 & \mathbb{Z}_2 \\
		0 & \mathbb{Z}_2 \\
	\end{array} \right] $. Note that  $X=\left(\left[\begin{array}{rr}
		1 & 0 \\
		0 & 0 \\
	\end{array} \right], \left[\begin{array}{rr}
		0 & 1 \\
		0 & 1 \\
	\end{array} \right]\right) \in l(A)$ and take $Y=\left(\left[\begin{array}{rr}
		1 & 1 \\
		0 & 1 \\
	\end{array} \right], \left[\begin{array}{rr}
		1 & 1 \\
		0 & 1 \\
	\end{array} \right]\right) \in R_1\times R_2$. Then \\ $XY=\left(\left[\begin{array}{rr}
		1& 1 \\
		0 & 0 \\
	\end{array} \right], \left[\begin{array}{rr}
		0 & 1 \\
		0 & 1 \\
	\end{array} \right]\right) \notin l(A)$. Thus, $l(A)$ is not an  ideal of $R_1\times R_2$. Observe that \\ $r(A)=\left\{\left(\left[\begin{array}{rr}
		x & y \\
		0 & 0 \\
	\end{array} \right], \left[\begin{array}{rr}
		0 & z \\
		0 & z \\
	\end{array} \right]\right)| x,y,z\in \mathbb{Z}_2 \right\}$. So, $P=\left(\left[\begin{array}{rr}
		1& 1 \\
		0 & 0 \\
	\end{array} \right], \left[\begin{array}{rr}
		0 & 1 \\
		0 & 1 \\
	\end{array} \right]\right) \in r(A)$. Then, $YP=\left(\left[\begin{array}{rr}
		1& 1 \\
		0 & 0 \\
	\end{array} \right], \left[\begin{array}{rr}
		0 & 0 \\
		0 & 1 \\
	\end{array} \right]\right) \notin r(A)$. So $r(A)$ is not an ideal of $R_1\times R_2$. Therefore, $R_1\times R_2$ is not pointwise semicommutative.
\end{example}

\begin{prop}\label{n2}
	Let $R$ be a pointwise semicommutative, and every nilpotent element is of index $\leq 2$. Then $R$ is 2-primal.
\end{prop}

\begin{proof}
	It is obvious that $P(R)\subseteq N(R)$. Let $w\in N(R)$. Then $w^2=0$. Since $R$ is pointwise semicommutative, $wRw=0\subseteq P(R)$ and so $w\in P(R)$. Therefore $P(R)=N(R)$. 
\end{proof}

Following \cite{lft}, an element $w$ of a ring $R$ is said to be \textit{clean} if $w$ is a sum of a unit, and an idempotent of $R$, $w$ is said to be \textit{exchange} if there exist $e\in E(R)$ such that $e\in wR$ and $1-e\in (1-w)R$. $R$ is said to be clean if every element of $R$ is clean, and $R$ is said to be exchange if every element of $R$ is exchange. It is well known that clean rings are exchange. 

\begin{prop}
	Let R be a pointwise semicommutative exchange ring, then R is clean. 
\end{prop}

\begin{proof}
	Let $w\in R$. Then there exists $e\in E(R)$ satisfying $e\in wR$ and $1-e\in (1-w)R$. So $e=wh$ and $1-e=(1-w)k$ for some $h,~k\in R$ such that $h=he$, $k=k(1-e)$. Then $(w-(1-e))(h-k)=wh-wk-(1-e)h+(1-e)k=wh+(1-w)k-(1-e)h-ek=1-(1-e)h-ek$. Since $R$ is pointwise semicommutative, either $r(e)$ or $l(e)$ is an ideal of $R$. If $r(e)$ is an ideal of $R$, then $0=hes(1-e)=hs(1-e)$, that is, $(1-e)hs\in N(R)$ for all $s\in R$ and $0=ek(1-e)=ek$. So $(1-e)h,~ek\in J(R)$. If $l(e)$ is an ideal of $R$, then $(1-e)h=0$ and $ek\in J(R)$. Therefore, $1-(1-e)h-ek$ is unit and by Proposition \ref{dlm} (\ref{df}), $w-(1-e)$ is unit. Thus, $w$ is clean.
\end{proof}

A ring $R$ is said to be semiperiodic (\cite{spdic}) if for each $w\in R\setminus(J(R)\cup C(R))$, $w^q-w^p\in N(R)$ for some integers $p$ and $q$ of opposite parity.

\begin{lemma}\label{nsjr}
	Let $R$ be a pointwise semicommutative ring. If $R$ is semiperiodic then $N(R)\subseteq J(R)$.
\end{lemma}

\begin{proof}
	Suppose $w~(\neq 0),~x\in R$ and $w^k=0$ for some positive integer $k$. If $wx\in J(R)$, then $wx$ is right quasi-regular. If $wx\in C(R)$, then $wx$ is nilpotent, and so $wx$ is right quasi-regular. Assume that $wx\notin J(R)\cup C(R)$. Then by \cite[Lemma 2.3(iii)]{spdic}, there exist a positive integer $p$ and $e\in E(R)$ such that $(wx)^p=(wx)^pe$ and $e=wy$ for some $y\in R$. Observe that $e=wy=ewy=ew(1-e)y+ewey=ew(1-e)y+ew^2y^2=...=\sum\limits_{i=1}^{k-1}ew^i(1-e)y^i+ew^ky^k=\sum\limits_{i=1}^{k-1}ew^i(1-e)y^i$. Since $R$ is pointwise semicommutative, $r(e)$ or $l(e)$ is an ideal of $R$. If $r(e)$ is an ideal of the $R$, then $ew^i(1-e)y^i=0$ for all $i$ and hence $e=0$. If $l(e)$ is an ideal of $R$ then $(1-e)re=0$ for any $r\in R$ and hence $ew^i(1-e)s\in N(R)$ for all $i$ and $s\in R$, that is, $ew^i(1-e)\in J(R)$. Hence $e=\sum\limits_{i=1}^{k-1}ew^i(1-e)y^i\in J(R)$, that is, $e=0$. Consequently, we obtain $(wx)^p=0$ and so $wx$ is  right quasi-regular. Thus, $w\in J(R)$.
\end{proof}

\begin{prop}
	Let $R$ be a pointwise semicommutative ring. If $R$ is semiperiodic, then:
	\begin{enumerate}
		\item $R/J(R)$ is commutative.
		\item $R$ is $NI$.
		\item $R$ is commutative whenever $J(R)\neq N(R)$.
	\end{enumerate}
\end{prop}

\begin{proof}
	\begin{enumerate}
		
		\item \label{sni}  By Lemma \ref{nsjr}, $N(R)\subseteq J(R)$. Write $\bar{R}=R/J(R)$ and let $\bar{w}\in \bar{R}$ with $\bar{w}^2=0$. Then by \cite[Lemma 2.6]{spdic}, $w^2\in J(R)\subseteq N(R)\cup C(R)$. If $w^2\in N(R)$, then $w\in N(R)\subseteq J(R)$ ( see Lemma \ref{nsjr}), that is, $\bar{w}=0$. Suppose $w^2\notin  N(R)$, then $w^2\in C(R)$. If $w\in Z(R)$, then $\bar{w}\bar{R}\bar{w}=0$. Since $\bar{R}$ is semiprime, $\bar{w}=0$. Assume, if possible, that $\bar{w}\notin C(\bar{R})$ then $w\notin J(R)\cup C(R)$. By \cite[Lemma 2.3(iii)]{spdic}, there exist a positive integer $p$ and $e\in E(R)$ such that $w^p=w^pe$ and $e=wy$ for some $y\in R$. Hence $e=ewy=ew(1-e)y+ewey=ew(1-e)y+ew^2y^2$. Since $R$ is pointwise semicommutative,  $e\in J(R)$, that is, $e=0$. This yields that $w^p=0$ and so $w\in N(R)\subseteq J(R)$, a contradiction. Therefore $\bar{w}\in C(\bar{R})$ and so $\bar{w}=0$. Thus, $\bar{R}$ is reduced. Since $R$ is semiperiodic, $\bar{R}$ is commutative (by \cite[Theorem 4.4]{spdic}).
		
		\item \label{k} Let $w,~h\in N(R)$ and $x\in R$. By Lemma \ref{nsjr}, $w-h,~wx\in J(R)$. By \cite[Lemma 2.6]{spdic}, $w-h,~wx\in N(R)\cup C(R)$. If $w-h,~wx\in N(R)$, then nothing to prove. Suppose $w-h,~wx\in C(R)$. Now, observe that $(w-h)w=w(w-h)$ and $(wx)^m=w^mx^m$ for all integer $m\geq 1$. This implies that $wh=hw$ and hence $w-h,~wx\in N(R)$. Hence $N(R)$ is an ideal.
		
		\item By \cite[Lemma 2.6]{spdic}, $J(R)=(J(R)\cap N(R))\cup (J(R)\cap C(R))$. Note that $R$ is NI (by \ref{k}), and $J(R)\cap N(R)$ and $J(R)\cap C(R)$ are additive subgroups of $R$, so $J(R)=J(R)\cap N(R)$ or $J(R)\cap C(R)$. This yields that $J(R)\subseteq N(R)$ or $J(R)\subseteq C(R)$. By hypothesis and Lemma \ref{nsjr}, $J(R)\subseteq C(R)$. Let $w\in R$. Suppose $w\notin C(R)$. Then $w\notin J(R)\cup C(R)$. Then there exist positive integers $p,q~(p\geq q)$ of opposite parity such that $w^p-w^q\in N(R)$. So $(w^p-w^q)^k=0$ for some $k\geq 1$. Then $((w-w^{p-q+1}))^k=0$, this gives $w-w^{p-q+1}\in N(R)\subseteq J(R)\subseteq C(R)$. By Herstein's Theorem \cite{spdic}, $R$ is commutative.
	\end{enumerate}
\end{proof}

The following examples show that the skew polynomial ring and the polynomial ring over a pointwise semicommutative ring need not be pointwise semicommutative.
\begin{example}

	\begin{enumerate}
		
		\item Let $D$ be a division ring and $R=D\bigoplus D$ with componentwise multiplication. Clearly, $R$ is reduced, so $R$ is pointwise semicommutative. Define $\sigma(h,w)=(w,h)$. Then $\sigma$ is an automorphism of $R$. Let $f(x)=(1,0)x\in R[x;\sigma]$. Observe that $f(x)^2=0$ but $f(x)xf(x)\neq 0$. Hence $R[x;\sigma]$ is not pointwise commutative. 
		
		\item Take $\mathbb{Z}_2$ as the field of integers modulo 2 and let $A=\mathbb{Z}_2[a_0,a_1,a_2,b_0, \\ b_1,b_2,c]$ be the free algebra of polynomials with zero constant terms in noncommuting indeterminates $a_0,a_1,a_2,b_0,b_1,b_2$ and $c$ over $\mathbb{Z}_2.$ Take an ideal $I$ of the ring $\mathbb{Z}_2+A$ generated by $a_0b_0, a_0b_1+a_1b_0, a_0b_2+a_1b_1+a_2b_0, a_1b_2+a_2b_1, a_1b_1,a_0rb_0, a_2rb_2, b_0a_0, b_0a_1+b_1a_0, b_0a_2+b_1a_1+b_2a_0, b_1a_2+b_2a_1,b_2a_2,b_0ra_0,b_2ra_2, (a_0+a_1+a_2)r(b_0+b_1+b_2),(b_0+b_1+b_2)r(a_0+a_1+a_2) ~and~ r_1r_2r_3r_4$ where $r,r_1,r_2,r_3,r_4\in A$. Take $R=(\mathbb{Z}_2+A)/I$. Then we have $R[x]\cong (\mathbb{Z}_2+A)[x]/I[x]$. By \cite[Example 2.1]{ext}, $R$ is reversible and hence pointwise semicommutative.  
		Observe that $(b_0+b_1x+b_2x^2)(a_0+a_1x+a_2x^2)\in I[x]$. But $(b_0+b_1x+b_2x^2)c(a_0+a_1x+a_2x^2)\notin I[x]$, since $b_0ca_1+b_1ca_0\notin I$. Hence $l(\overline{(a_0+a_1x+a_2x^2)})$ is not an ideal of $R[x]$. Again,  $(a_0+a_1x+a_2x^2)(b_0+b_1x+b_2x^2)\in I[x]$. But  $(a_0+a_1x+a_2x^2)c(b_0+b_1x+b_2x^2)\notin I[x]$, since $a_0cb_1+a_1cb_0\notin I$. Thus, $r(\overline{(a_0+a_1x+a_2x^2)})$ is also not an ideal of $R[x]$. Therefore, $R[x]$ is not pointwise semicommutative. 
	\end{enumerate}
	
\end{example}
\begin{lemma}\label{loc}
	Let $R$ be a ring and $\Delta $ be a multiplicatively closed subset of $R$ consisting of central non-zero divisors. For any $u^{-1}a \in  \Delta ^{-1} R$, $l(u^{-1}a)$ is an ideal of $\Delta ^{-1} R$ if and only if $l(a)$ is an ideal of $R$ and $r(u^{-1}a)$ is an ideal of $\Delta ^{-1} R$ if and only if $r(a)$ is an ideal of $R$.
\end{lemma}
\begin{proof}
	Easy to prove.
\end{proof}
\begin{prop}
	Let $R$ be a ring and $\Delta$ be a multiplicatively closed subset of $R$ consisting of central non-zero divisors. Then $R$ is pointwise semicommutative if and only if $\Delta ^{-1} R$ is pointwise semicommutative.
\end{prop}
\begin{proof}
	Suppose $R$ is pointwise semicommutative. Let $u^{-1}a\in \Delta^{-1} R$. Then, either $l(a)$  or $r(a)$ is an ideal of $R$. By the Lemma \ref{loc}, either $l(u^{-1}a)$ or $r(u^{-1}a)$ is an ideal of $ \Delta^{-1} R$. Thus,  $\Delta^{-1} R$ is pointwise semicommutative. Observe that the converse is trivial. 
\end{proof}
\begin{cor}
	$R[x]$ is pointwise semicommutative if and only if $R[x,x^{-1}]$ is so.
\end{cor}
\begin{proof}
	Note that $R[x,x^{-1}]=\Delta ^{-1}R[x]$, where $\Delta =\{1,x,x^2,...\}$. Hence the result follows.
\end{proof}

A left $R$-module $M$ is said to be $Wnil-injective$ if for any $w~(\neq 0)\in N(R)$, there exists a positive integer $m$ such that $w^m\neq 0$ and any  $R$-homomorphism $\Psi:Rw^m\rightarrow M$ extends to one from $R$ to $M$. In order to probe some properties of pointwise semicommutative rings, we investigate Wnil-injective modules over a pointwise semicommutative ring in the following. 

\begin{prop}
	Let $R$ be a pointwise semicommutative ring, and every simple singular left $R$-module is Wnil-injective then: 
	\begin{enumerate}
		\item \label{non} $R$ is left non-singular.
		
		\item  $R$ is left weakly regular whenever $R$ is left MC2.
		
		\item R is reduced if any $e\in E(R)$, $er=ere$ for all $r\in R$.
		
	\end{enumerate}
\end{prop}
\begin{proof} 
	\begin{enumerate}
		\item Assume that $Z(_RR)\neq 0$. Then there exists $w~(\neq 0)\in Z(_RR)$ such that $w^2=0$. So, $l(w)\subseteq M$ for some maximal left ideal $M$ of $R$. Since $w\in Z(_RR)$, $M$ is essential. Now, define an $R$-homomorphism $\Psi:Rw\rightarrow R/M$ via. $\Psi(rw)=r+M$.  By hypothesis, $R/M$ is Wnil-injective and  so there exists $h\in R$ with $1-wh\in M$. Since $R$ is pointwise semicommutative and $w^2=0$, $whw=0$, that is, $wh\in l(w)$, which further implies that $1\in M$, a contradiction. Therefore $Z(_RR)=0$.

		\item  Suppose there is an element $w\in R$ such that $RwR+l(w)\neq R$. So, $RwR+l(w)\subseteq M$ for some maximal left ideal $M$ of $R$. If $M$ is not essential in $_RR$, then $M=Re=l(1-e)$ for some $e\in E(R)$. As $R(1-e)\cong R/l(1-e)=R/M$ is a simple left $R$-module, $R(1-e)$ is a minimal left ideal of $R$.  By Proposition \ref{dlm} (\ref{lma}), $R$ is a left min-abel ring. Since $R$ is a left $MC2$ ring, $1-e\in Z(R)$ by \cite[Theorem 1.8]{mc2}. As $w\in RwR+l(w)\subseteq M=l(1-e),~w(1-e)=0=(1-e)w$. So $1-e\in l(w)\subseteq M=l(1-e)$, a contradiction. Therefore, $M$ is essential left ideal of $R$. Thus, $R/M$ is Wnil-injective. As in the proof of (\ref{non},) $1-wh\in M$ for some $h\in R$. Since $wh\in RwR\subseteq M$, $1\in M$, a contradiction. Therefore, $RwR+l(w)=R$ for any $w\in R$, that is, $RwRw=R$. Hence $R$ is a left weakly regular ring.
		
		\item Suppose there exists $w~(\neq 0)\in R$ satisfying $w^2=0$. Then $l(w)\subseteq M$ for some maximal left ideal $M$ of $R$. If $M$ is not essential, then $M=l(e)$ for some $0\neq e\in E(R)$. So $we=0$ and  by hypothesis, $ew=ewe=0$. This implies that $e\in l(w)\subseteq M=l(e)$, a contradiction. Hence $M$ is essential, and $R/M$ is simple singular left $R$-module.
		As in the proof of (\ref{non}), $1-wh\in M$ for some $h\in R$. Since $R$ is pointwise semicommutative and $w^2=0$, $wh\in l(w)\subseteq  M$. This implies that $1\in M$, a contradiction. Therefore, $w=0$. 
		
	\end{enumerate}
\end{proof}

$R$ is called a \textit{left} (\textit{right}) $SF$  if all simple left (right) $R$-modules are flat.  \cite[Remark 3.13]{rnsf} shows that $R$ is strongly regular whenever $R$ is a reduced left SF ring. We extend this result as follows.

\begin{prop}
	Let R be a pointwise semicommutative ring. If R is left SF, then R is strongly regular.
\end{prop}
\begin{proof}
	By \cite[Proposition 3.2]{rnsf}, $R/J(R)$ is left SF. Let $w^2\in J(R)$ such that $w\notin J(R)$. Assume, if possible, $Rr(w)+J(R)=R$, then \\ $1=x+\sum\limits^{finite}r_is_i,~ x\in J(R),~ r_i\in R,~ s_i\in r(w)$. Then $w=xw+\sum\limits^{finite}r_is_iw$. Take $t_i=s_iw$. So $t_i^2=0$. Since $R$ is pointwise semicommutative, $t_iRt_i=0$. Suppose $t_i\notin J(R)$. Then $M+Rt_i=R$ for some maximal left ideal $M$ of $R$ with $t_i\notin M$. So $m+pt_i=1,~m\in M,~p\in R$. This yields that $(1-m)^2=0$, that is, $1\in R$. This is a contradiction. Therefore, $t_i\in J(R)$. This further yields that $w\in J(R)$, a contradiction. Hence $Rr(w)+J(R)\neq R$. There exist some maximal left ideal $H$ satisfying $Rr(w)+J(R)\subseteq H$. Note that $w^2\in H$. By \cite[Lemma 3.14]{rnsf}, $w^2=w^2x$ for some $x\in H$, that is, $w-wx\in r(w)\subseteq H$. So $w\in H$. Hence there exists $y\in H$ satisfying $w=wy$, that is, $1-y\in r(w)\subseteq H$. This implies that $1\in H$, a contradiction. Therefore $R/J(R)$ is reduced. Therefore by \cite[Remark 3.13]{rnsf}, $R/J(R)$ is strongly regular. This implies that $R$ is left quasi-duo, and hence by \cite[Theorem 4.10]{rnsf}, $R$ is strongly regular.
\end{proof}

\end{document}